\documentclass[12pt, twoside, leqno]{article}
\usepackage{amsmath,amsthm}
\usepackage{amssymb,latexsym}
\usepackage{enumerate}
\usepackage{bbm}

\newtheorem{thm}{Theorem}[section]

\newtheorem{lem}[thm]{Lemma}

\newtheorem*{xthm}{Theorem}

\theoremstyle{definition}

\numberwithin{equation}{section}

\multlinetaggap\parindent

\newcommand{\textnum}{\setbox0=\hbox{(1.1)\ }
\leftmargini\wd0
\setlength{\labelsep}{5pt}
\def\labelenumi{\rm(\theenumi)}
\numberwithin{enumi}{section}}

\newcommand{\bbtextnum}{\setbox0=\hbox{$(3.10)_b$\ }
\leftmargini\wd0
\setlength{\labelsep}{5pt}
\def\labelenumi{\rm(\theenumi)}
\numberwithin{enumi}{section}}

\newcommand{\dbtextnum}{\setbox0=\hbox{(2.12)\ }
\leftmargini\wd0
\setlength{\labelsep}{5pt}
\def\labelenumi{\rm(\theenumi)}
\numberwithin{enumi}{section}}

\newcommand{\bnumer}{\textnum\begin{enumerate}\setcounter{enumi}{\arabic{equation}}}
\newcommand{\bbnumer}{\bbtextnum\begin{enumerate}\setcounter{enumi}{\arabic{equation}}}
\newcommand{\dbnumer}{\dbtextnum\begin{enumerate}\setcounter{enumi}{\arabic{equation}}}
\newcommand{\enumer}{\end{enumerate}\setcounter{equation}{\arabic{enumi}}}

\def\<{\langle}
\def\>{\rangle}
\renewcommand{\]}{\mathopen]}
\renewcommand{\[}{\mathclose[}

\newcommand{\symC}{\mathbb C}
\newcommand{\symN}{\mathbb N}
\newcommand{\symR}{\mathbb R}

\newcommand{\calD}{\mathcal D}
\newcommand{\calE}{\mathcal E}
\newcommand{\calN}{\mathcal N}
\newcommand{\calO}{\mathcal O}
\newcommand{\calR}{\mathcal R}
\newcommand{\calS}{\mathcal S}

\DeclareMathOperator{\hRe}{Re}
\DeclareMathOperator{\loc}{loc}

\font\eus=eusm10 at12truept
\newcommand{\eusF}{\hbox{\eus F}}

\DeclareMathOperator{\supp}{supp}
\DeclareMathOperator{\dist}{dist}

\multlinetaggap\parindent

\begin{document}

\baselineskip=17pt

\title{Fundamental solutions of evolutionary PDOs
and rapidly decreasing distributions}
\author{Jan Kisy\'nski\\
Institute of Mathematics, Polish Academy of Sciences\\
\'Sniadeckich 8, 00-956 Warszawa, Poland\\
E-mail: jan.kisynski@gmail.com}

\date{}

\maketitle

\renewcommand{\thefootnote}{}

\footnote{2010 \emph{Mathematics Subject Classification}: Primary 35E05, 46F99.}

\footnote{\emph{Key words and phrases}: PDO with constant coefficients, Petrovski\u\i\ condition,
rapidly decreasing distributions, slowly increasing functions.}

\begin{abstract}
Let $P(\partial_0,\partial_1,\ldots,\partial_n)$ be a PDO on $\symR^{1+n}$
with constant coefficients. It is proved that
\begin{itemize}
\item[(i)] the real parts of the $\lambda$-roots of  the polynomial
$P(\lambda,i\xi_1,\ldots,i\xi_n)$  are bounded from above when
$(\xi_1,\ldots,\xi_n)$ ranges over $\symR^n$
\end{itemize}
if and only if
\begin{itemize}
\item[(ii)] $P$ has a fundamental solution with support in
$H_+=\{(x_0,x_1,\allowbreak\ldots, x_n)\in \symR^{1+n}:x_0\ge0\}$ having some special
properties expressed in terms of the L.~Schwartz space $\calO^{\prime}_C$ of rapidly decreasing distributions.
\end{itemize}
Moreover, it is proved that the fundamental solution with support in $H_+$ having these special properties is unique. 
\end{abstract}

\section{Introduction and the main result}\label{sec1}

\subsection{Rapidly decreasing distributions}\label{subsec1.1}
By Theorem IX in Sec.~VII.5 of L.~Schwartz's book \cite{11}, for every distribution $T\in\calD^\prime (\symR^n)$ the following two conditions are equivalent:
\bnumer
\item\label{eq1.1} $T\ast \varphi\in \calS(\symR^n)$ for every $\varphi\in\calD(\symR^n)$,
\item\label{eq1.2} for every $k\in\symN_0$ there is $m_k\in\symN_0$
such that $T=\sum_{|\alpha|\le m_k}\partial^\alpha F_{k,\alpha}$
where, for every multiindex $\alpha=(\alpha_1,\ldots,\alpha_n)\in\symN_0$
of length $|\alpha|=\alpha_1+\cdots+\alpha_n\le m_k$, $F_{k,\alpha}$
is a continuous function on $\symR^n$ such that $\sup_{x\in\symR^n}(1+|x|)^k|F_{k,\alpha}(x)|<\infty$.
\enumer
In the above, and everywhere in the following,
$\partial^\alpha=\partial^{\alpha_1}_1\ldots\partial^{\alpha_n}_n$ where $\partial_1,\ldots,\partial_n$ are partial derivatives of the first order not multiplied by any factor. Each of the conditions \eqref{eq1.1}, \eqref{eq1.2} is satisfied if and only if the distribution $T$ is \emph{rapidly decreasing}, where the definition of rapid decrease, due to L.~Schwartz, refers to the notion of boundedness of a distribution. The space of rapidly decreasing distributions on $\symR^n$ is denoted by $\calO^{\prime}_C(\symR^n)$. 
From \eqref{eq1.2} it follows that
\begin{equation}\label{eq1.3dod}
\mbox{whenever}\quad T\in\calO^{\prime}_C(\symR^n)\ \mbox{and}\ \varphi\in C^{\infty}_{b}(\symR^n),\quad
\mbox{then}\quad\varphi T\in\calO^{\prime}_C(\symR^n).
\end{equation}

It is clear from \eqref{eq1.2} that $\calO^{\prime}_C(\symR^n)\subset
\calS^\prime(\symR^n)$, so that the Fourier transform $\eusF T$
makes sense for every $T\in \calO^{\prime}_C(\symR^n)$. By Theorem~XV in Sec.~VII.8 of \cite{11},
\begin{equation}\label{eq1.3}
\eusF \calO^{\prime}_C(\symR^n)= \calO_M(\symR^n),
\end{equation}
where $\calO_M(\symR^n)$ denotes the space of \emph{infinitely differentiable slowly increasing functions} on $\symR^n$. Recall that $\phi\in \calO_M(\symR^n)$ if and only if $\phi\in C^{ \infty}(\symR^n)$ and for every $\alpha\in\symN^n_0$ 
there is $m_\alpha\in\symN_0$ such that
$$
\sup_{\xi\in\symR^n}(1+|\xi|)^{-m_\alpha}|\partial^\alpha\phi(\xi)|<\infty.
$$

Complete proofs of theorems about $\calO^{\prime}_C(\symR^n)$ and
$\calO_M(\symR^n)$ needed in the present paper may be found in \cite{7}.

\subsection{The main result}\label{subsec1.2} 
Our object of interest  will be the differential operator $P(\partial_0,\partial_1,\ldots,\partial_n)$ on $\symR^{1+n}=\{(x_0,x_1,\ldots,x_n):x_\nu\in\symR$ for $\nu=0,\ldots,n\}$ with constant coefficients, and the associated polynomial $P(\lambda,i\xi_1,\ldots,i\xi_n)$ defined on $\symC\times\symR^n$.
A distribution $N$ on $\symR^{1+n}$ such that
$$
PN\equiv\delta
$$
is called a {\it fundamental solution} for the operator $P$. Let
$$
H_+=\{(x_0,x_1,\ldots, x_n)\in \symR^{1+n}:x_0\ge0\}.
$$
If there exists a fundamental solution $N$ for $P$ such that $\supp N\subset H_+$, then  the operator $P$ is said
to be {\it evolutionary} with respect to~$H_+$.
For every fixed $\lambda\in\symC$ let $e_{-\lambda}$ be the function on $\symR^{1+n}$ given by
$e_{-\lambda}(x_0,x_1,\ldots, x_n)=\exp(-\lambda x_0)$ for
$(x_0,x_1,\ldots, x_n)\in\symR^{1+n}$. For $\vartheta\in \calD(\symR)$, denote by $\vartheta_0$ the function on $\symR^{1+n}$ such that $\vartheta_{0}
(x_0,x_1,\ldots, x_n)=\vartheta(x_0)$. Let
\begin{multline*} \indent
\calO'_{\rm LOC}(H_+)=\{T\in\calD'(\symR^{1+n}):\supp T\subset H_+,\\
\vartheta_0 T\in\calO'_C(\symR^{1+n})\mbox{ for every }\vartheta\in \calD(\symR)\}.\indent
\end{multline*}

\begin{xthm}
Let $P(\partial_0,\partial_1,\ldots,\partial_n)$
be the differential operator on $\symR^{1+n}$ with constant coefficients. Let
\begin{multline*}\indent
\omega_0=\sup\{\hRe\lambda:\lambda\in\symC\text{ and there is }
(\xi_1,\ldots,\xi_n)\in\symR^n\\ 
\text{ such that }P(\lambda,i\xi_1,\ldots,i\xi_n)=0\}.\indent
\end{multline*}
Then the following two conditions are equivalent:
\begin{itemize}
\item[\rm(i)] $\omega_0<\infty$,
\item[\rm(ii)] the differential operator $P(\partial_0,\partial_1,\ldots,\partial_n)$ has a fundamental solution $N$
belonging to $\calO'_{\rm LOC}(H_+)$.
\end{itemize}
Furthermore, if {\rm(i)} and {\rm(ii)} are satisfied, then the fundamental solution $N$ as in {\rm(ii)}
is unique and satisfies
\begin{itemize}
\item[\rm(iii)] $\omega_0=\inf\{\hRe\lambda:\lambda \in \symC,\, e_{-\lambda}N
\in\calO^{\prime}_C(\symR^{1+n})\}$, and $e_{-\lambda}N\in\calO^{\prime}_C(\symR^{1+n})$
 whenever $\hRe\lambda>\omega_0$.
\end{itemize}
\end{xthm}

\subsection{Remarks}\label{subsec1.3} 
Condition (i) can be called the Petrovski\u\i\ condition because it first appeared in I.~G.~Petrovski\u\i's paper \cite{8}. Namely, in \cite{8}, in the footnote on p.~24, it was conjectured that, if  the polynomial
$P(\lambda,i\xi_1,\ldots,i\xi_n)$ 
is unital with respect to $\lambda$, then this condition 
is equivalent to a certain formally weaker condition also concerning the
$\lambda$-roots of $P(\lambda,i\xi_1,\ldots,i\xi_n)$. 
The validity of this conjecture was proved by L.~G{\aa}rding in~\cite{3}. I.~G.~Petrovski\u\i\ noticed the significance of smooth slowly increasing functions for the theory of evolutionary PDEs with constant coefficients. L.~Schwartz explained in \cite{10} how the results of Petrovski\u\i\ may be elucidated by placing them in the framework of rapidly decreasing distributions and smooth slowly increasing functions. (Condition (i) was not mentioned in \cite{10}; notice that \cite{10} was earlier than \cite{3}.)
 
L.~H\"ormander proved in \cite{5} that if $P(\zeta_0,\zeta_1,\ldots,\zeta_n)$ is a polynomial of $1 +n$
complex variables, then the following two conditions are equivalent:
\begin{itemize}
\item[${\rm(i)}^*$] there are constants $A\in\]-\infty,\infty\[$ and $r\in\]0,\infty\[$ 
such that
$$
\inf\{\hRe F(\zeta_1,\ldots,\zeta_n):(\zeta_1,\ldots,\zeta_n)\in B_{i\xi_1,\ldots,i\xi_n;r}\}\le A
$$ 
for every $(\xi_1,\ldots,\xi_n)\in \symR^{n}$ and every function $F$ holomorphic in the ball
$$
B_{i\xi_1,\ldots,i\xi_n;r}=\Big\{(\zeta_1,\ldots,\zeta_n)\in \symC^n:
\sum_{\nu=1}^n|\zeta_{\nu}-i\xi_{\nu}|^2<r^2\Big\}
$$
 such that $P(F(\zeta_1,\ldots,\zeta_n),\zeta_1,\ldots,\zeta_n)=0$ in $B_{i\xi_1,\ldots,i\xi_n;r}$,
\item[${\rm(ii)}^*$] the differential operator $P(\partial_0,\partial_1,\ldots,\partial_n)$ has a fundamental solution with support in~$H_+$.
\end{itemize}
The equivalence ${\rm(i)}^*{\Leftrightarrow}{\rm(ii)}^*$ was reproved in Sec.~12.8 of~\cite{6}. The fundamental solution occurring in ${\rm(ii)}^*$ need not be unique.
It is non-unique if ${\rm(i)}^* $ holds and the boundary of $H_+$ is characteristic for 
$P(\partial_0,\partial_1,\ldots,\partial_n)$. Obviously (i) implies ${\rm(i)}^*$. Furthermore, as indicated in \cite{5}, the operator $\partial_0-i(\partial_1+1)^2$ satisfies ${\rm(i)}^*$ but does not satisfy (i). Therefore condition ${\rm(i)}^*$ is essentially weaker than (i).

Let us stress that in \cite{5}, and in the present paper, the largest power of $\lambda$ in $P(\lambda,i\xi_1,\ldots,i\xi_n)$ is multiplied by a polynomial of
$\xi_1,\ldots,\xi_{n}$ which, in contrast to the assumption (5) in Sec.~3.10 of \cite{9},  may vanish for some $(\xi_1,\ldots,\xi_n)\in\symR^n$.

\section{Existence of a fundamental solution
satisfying (ii) and (iii)}\label{sec2}

\subsection{Application of the Tarski--Seidenberg theorem}\label{subsec2.1} 
We are going to prove that if (i) holds, then the differential operator\break
$P(\partial_0,\partial_1,\ldots,\partial_n)$
has a fundamental solution $N$ satisfying the conditions (ii) and (iii). So, suppose that (i) holds and let
$$
\calN=\{(\sigma,\xi_0,\ldots,\xi_n)\in\symR^{2+n}:P(\sigma+i\xi_0,i\xi_1,\ldots,i\xi_n)=0\}.
$$
Then $\calN\subset\{(\sigma,\xi_0,\ldots,\xi_n)\in\symR^{2+n}:
\sigma\le\omega_0\}$, and hence, by Theorem A.3 from the Appendix to \cite{12} or by Theorem 3.2 of \cite{4}$^{*)}$\footnote{$^{*)}$Following the idea of L.~H\"ormander, these theorems are deduced from the Tarski--Seidenberg theorem about projections of semi-algebraic sets.}, 
there are $c,\mu,\mu'\in\]0,\infty\[$ such that whenever $\sigma\in\]\omega_0,\infty\[$, 
and $(\xi_0,\ldots,\xi_n)\in\symR^{1+n}$, then
\begin{multline}\label{eq2.1}
|P(\sigma+i\xi_0,i\xi_1,\ldots,i\xi_n)|\ge c(\dist((\sigma,\xi_0,\ldots,\xi_n);\calN))^{\mu}\\
\hfill\cdot(1+(\sigma^2+\xi_0^2+\cdots+\xi_n^2)^{1/2})^{-\mu'}\\
\ge c(\sigma-\omega_0)^{\mu}(1+|\sigma+i\xi_0|+(\xi_1^2+\cdots+\xi_n^2)^{1/2})^{-\mu'}.
\end{multline}

\subsection{{\spaceskip.33em plus.22em minus.17em The slowly increasing functions $\widehat{N_{\sigma}}$ and the rapidly decreasing distributions $N_{\sigma}$}}\label{subsec2.2}
For every $\sigma\in\]\omega_0,\infty\[$ let
\begin{equation}\label{eq2.3}
\widehat{N_{\sigma}}(\xi_0,\ldots,\xi_n)=(P(\sigma+i\xi_0,
i\xi_1,\ldots,i\xi_n))^{-1}
\end{equation} 
for $(\xi_0,\ldots,\xi_n)\in\symR^{1+n}$. Then, for every multiindex $\alpha\in\symN^{1+n}$,
$$
\partial^{\alpha}\widehat{N_{\sigma}}(\xi_0,\ldots,\xi_n)
=(P(\sigma+i\xi_0,i\xi_1,\ldots,i\xi_n))^{-1-|\alpha|}Q_{\alpha}(\sigma,\xi_0,\ldots,\xi_n)
$$
where $Q_{\alpha}$ is a polynomial. Consequently, \eqref{eq2.1} implies that
\begin{equation}\label{eq2.4}
\widehat{N_{\sigma}}\in\calO_M(\symR^{1+n})\ \quad\mbox{for every }
\sigma\in\]\omega_0,\infty\[.
\end{equation}
Let
\begin{equation}\label{eq2.5}
N_{\sigma}=\eusF^{-1}\widehat{N_{\sigma}}
\end{equation}
where $\eusF$ denotes the Fourier transformation on $\symR^{1+n}$ 
such that 
\begin{align*}
(\eusF\varphi)(\xi_0,\ldots,\xi_n)&=\widehat\varphi(\xi_0,\ldots,\xi_n)\\
&=\mathop{\int\cdots\int}\limits_{\symR^{1+n}}e^{-i\sum_{\nu=0}^n x_\nu\xi_\nu}
\varphi(x_0,\ldots,x_n)\,d x_0\ldots\, dx_n
\end{align*}
for $\varphi\in\calS(\symR^{1+n})$, and $\eusF$ is extended onto $\calS'(\symR^{1+n})$ by duality.
From \eqref{eq1.3} and \eqref{eq2.4} it follows that
\begin{equation}\label{eq2.6}
N_{\sigma}\in\calO^{\prime}_C(\symR^{1+n}) \ \quad\mbox{for every }
\sigma\in\]\omega_0,\infty\[.
\end{equation}
Furthermore, from \eqref{eq2.3} it follows that
\bnumer
\item\label{eq2.7}
if $\sigma\in\]\omega_0,\infty\[$
then $N_{\sigma}$ is a fundamental solution for the differential operator
$P(\sigma+\partial_0,\partial_1,\ldots,\partial_n)$.
\enumer
Take 
$\sigma\in\]\omega_0,\infty\[$, and consider the distribution $e_{\sigma}N_{\sigma}
\in \calD'(\symR^{1+n})$. By the Parseval equality, for every
$\varphi\in\calD(\symR^{1+n})$ one has
\begin{multline*}
\<e_\sigma N_\sigma,\varphi\>=\<N_\sigma, e_\sigma\varphi\>=(2\pi)^{-1-n}
\<\widehat{N_\sigma}, \widehat{e_\sigma\varphi}^{\vee}\>\\
=(2\pi)^{-1-n}\mathop{\int\!\!\cdots\!\!\int}\limits_{\symR^{1+n}} (\widehat{e_\sigma\varphi}(-\xi_0,\ldots,-\xi_n)(P
(\sigma+i\xi_0,i\xi_1,\ldots,i\xi_n))^{-1}\,d\xi_0\ldots d\xi_n.
\end{multline*}
For every $\varphi\in\calD(\symR^{1+n})$ the Fourier integral
$$
\widehat{\varphi}(\zeta_0,\ldots,\zeta_n)=\mathop{\int\!\!\cdots\!\!\int}\limits_{\symR^{1+n}}
e^{-i\sum^{n}_{\nu=0}x_{\nu}\zeta_{\nu}}\varphi(x_0,\ldots,x_n)\,dx_0\ldots dx_n
$$
makes sense for $(\zeta_0,\ldots,\zeta_n)\in\symC^{1+n}$ and defines the holomorphic extension of $\widehat{\varphi}$ from $\symR^{1+n}$
onto $\symC^{1+n}$. This holomorphic extension satisfies
$$
\widehat{e_\sigma\varphi}(\zeta_0,\ldots,\zeta_n)=\widehat{\varphi}(\zeta_0+i\sigma,
\zeta_1,\ldots,\zeta_n).
$$
Consequently, whenever $\varphi\in\calD(\symR^{1+n})$ and $\sigma\in\]\omega_0,\infty\[$, then
\begin{multline}\label{eq2.8}
\<e_\sigma N_\sigma,\varphi\>=(2\pi)^{-1-n}\mathop{\int\!\!\cdots\!\!\int}\limits_{\symR^{1+n}}\widehat\varphi
(-\xi_0+i\sigma,-\xi_1,\ldots,-\xi_n)\\
\cdot P(\sigma+i\xi_0,\xi_1,\ldots,\xi_n)^{-1}\,d\xi_0\ldots d\xi_n.
\end {multline}
Integration by parts shows that whenever $\varphi\in\calD(\symR^{1+n})$ and $l\in \symN$, then
\begin{multline}\label{eq2.9}
(1+|\xi_0-i\sigma|^l+|\xi_1|^l+\cdots+|\xi_n|^l)|\widehat{\varphi}
(-\xi_0+i\sigma,-\xi_1,\ldots,-\xi_n)|\\
\le\Big(\|\varphi\|_{L^{1}(\symR^{1+n})}+\sum^{n}_{\nu=0}\|\partial^l_\nu\varphi\|_{L^{1}(\symR^{1+n})}\Big)
\exp(H_{\varphi}(\sigma))
\end {multline} 
for every $\sigma,\xi_0,\ldots,\xi_n\in\symR$ where
\begin{equation}\label{eq2.10}
H_\varphi(\sigma)=\sup\{\sigma x_0:(x_0,\ldots,x_n)\in\supp\varphi\}.
\end{equation}
From \eqref{eq2.1}, \eqref{eq2.8}--\eqref{eq2.10} and the Cauchy integral theorem it follows that
\bbnumer
\item\label{eq2.12}
the distribution $e_\sigma N_\sigma\in\calD'(\symR^{1+n})$ does not depend on $\sigma$ provided that $\sigma\in\]\omega_0,\infty\[$,
\item\label{eq2.13}
$\lim_{\sigma\to\infty}\<e_\sigma N_\sigma,\varphi\>\!=\!0$ whenever
$\varphi\in\calD(\symR^{1+n})$ and $\supp \varphi\!\subset\!\symR^{1+n}\setminus\nobreak H_+$.
\enumer

\subsection{The fundamental solution $N$}\label{subsec2.3} Thanks to \eqref{eq2.12} we may define the distribution $N\in\calD'(\symR^{1+n})$ by the equality
\begin{equation}\label{eq2.14}
N=e_\sigma N_\sigma\ \quad\mbox{for every }\sigma\in\]\omega_0,\infty\[.
\end{equation}
From \eqref{eq2.13} it follows that
\begin{equation}\label{eq2.15}
\supp N\subset H_+.
\end{equation}
 For every $\sigma\in\symR$ let
\begin{equation}\label{eq2.16}
S_\sigma=P(\sigma+\partial_0,\partial_1,\ldots,\partial_n)\delta.
\end{equation}


Since $(-\partial_{0})^{k}(e_{-\sigma}\varphi)=e_{-\sigma}(\sigma-\partial_{0})^{k}\varphi$,
it follows that
\begin{equation}\label{eqnow2.15}
P(-\partial_0,-\partial_1,\ldots,-\partial_n)(e_{-\sigma}\varphi)=e_{-\sigma}
P(\sigma-\partial_0,-\partial_1,\ldots,-\partial_n)\varphi
\end{equation}
for every $\sigma\in\symR$ and $\varphi\in \calD(\symR^{1+n})$.
From \eqref{eqnow2.15} one infers that 
\begin{align*}
 \<S_0, e_{-\sigma}\varphi\>
&=[P(-\partial_0,-\partial_1 ,\ldots,-\partial_n)
(e_{-\sigma}\varphi)](0)\\
&=[e_{\sigma}P(-\partial_0,-\partial_1 ,\ldots,-\partial_n)
(e_{-\sigma}\varphi)](0)\\
&=[P(\sigma-\partial_0,-\partial_1 ,\ldots,-\partial_n)\varphi](0)
=\<S_{\sigma},\varphi\>,
\end{align*}
proving that
\begin{equation}\label{eqnow2.16}
S_\sigma=e_{-\sigma} S_0\ \quad\mbox{for every }\sigma\in \symR.
\end{equation}
From \eqref{eq2.7}, \eqref{eq2.14} and \eqref{eqnow2.15} it follows that whenever $\sigma
\in\]\omega_0,\infty\[$, then
\begin{equation*}
PN=S_0\ast N=(e_{\sigma}S_\sigma)\ast(e_{\sigma}N_\sigma)
=e_{\sigma}(S_\sigma\ast N_\sigma)=e_{\sigma}\delta=\delta,
\end{equation*}
so that
\begin{equation}\label{eq2.18}
\mbox{$N$ is a fundamental solution for the operator $P$.}
\end{equation}
Above we have used the fact that whenever $T,U\in\calD'(\symR^{1+n})$,
$\sigma\in\symR$, and one of $T,U$ has compact support, then $e_{\sigma} (T*U)
=(e_{\sigma}T)*(e_{\sigma}U)$. This is true under the additional assumption that
$T,U\in L^1_{\loc}(\symR^{1+n})$, and this case implies the general assertion by regularization.

\subsection{Properties of $N$}\label{subsec2.4}
If $\vartheta\in\calD(\symR)$ and $\sigma\in\]\omega_0,\infty\[$, then $\vartheta_0 e_{\sigma}$ is bounded on $\symR^{1+n}$ together with all its partial derivatives, so that, by \eqref{eq1.3dod},
$\vartheta_0 N=(\vartheta_0 e_{\sigma})N_\sigma\in\calO^{\prime}_C(\symR^{1+n})$ because $N_\sigma\in\calO^{\prime}_C(\symR^{1+n})$.
Hence, by \eqref {eq2.15},
\begin{equation}\label{eqnow2.18}
N\in\calO'_{\rm LOC}(H_+).
\end{equation}
The relations \eqref{eq2.18}  and \eqref{eqnow2.18} show that (i) implies (ii). We are going to prove that
 $N$  defined by \eqref{eq2.14} satisfies (iii). To this end, take $\lambda\in\symC$  such that $\hRe \lambda\in\]\omega_0,\infty\[$.  Let 
$\sigma=\frac12(\omega_0+\hRe\lambda)$. Then
$e_{-\lambda}N=e_{\sigma-\lambda} N_\sigma\in\calO^{\prime}_C(\symR^{1+n})$ because $N_{\sigma}\in\calO^{\prime}_C(\symR^{1+n})$, $\supp N_\sigma\subset H_+$, 
and $e_{\sigma-\lambda}$ is bounded together with all its partial derivatives on the set
$\{(x_0,\ldots,x_n)\in\symR^{1+n}:x_0>-1\}$. 
It remains to prove that
\begin{equation}\label{eq2.20}
\mbox{if $\lambda\in\symC$ and $e_{-\lambda}N\in\calO^{\prime}_C(\symR^{1+n})$, then $\hRe\lambda\ge\omega_0$.}
\end{equation}
So, suppose that $\lambda\in\symC$ and $e_{-\lambda}N\in\calO^{\prime}_C(\symR^{1+n})$. Take any $\sigma\in\]\hRe\lambda,\infty\[$. Since 
$e_{\lambda-\sigma}$ is bounded on $\{(x_0,\ldots,x_n)\in\symR^{1+n}:
x_0>-1\}$ together with all its partial derivatives, it follows by \eqref{eq1.3dod} that $e_{-\sigma}N
=e_{\lambda-\sigma}(e_{-\lambda}N)\in\calO^{\prime}_C(\symR^{1+n})$. Furthermore
$$
S_\sigma*(e_{-\sigma}N)=(e_{-\sigma}S_0)*(e_{-\sigma}N)=e_{-\sigma}(S_0*N)
=e_{-\sigma}\delta=\delta. 
$$
Let $\phi=\eusF(e_{-\sigma}N)$. Then $\phi\in\calO_M(\symR^{1+n})$ and
\begin{multline*} 
P(\sigma+i\xi_0,i\xi_1,\ldots,i\xi_n)
\cdot\phi(\xi_0,\ldots,\xi_n)
=[\eusF(S_\sigma*(e_{-\sigma}N))](\xi_0,\ldots,\xi_n)=1
\end{multline*} 
for every $(\xi_0,\ldots,\xi_n)\in\symR^{1+n}$. This implies that
$P(\sigma+i\xi_0,i\xi_1,\ldots,i\xi_n)\ne0$
for every
$(\xi_0,\ldots,\xi_n)\in\symR^{1+n}$. Since this is true for every $\sigma\in\]\hRe\lambda,\infty\[$, it follows that $\hRe\lambda\ge\omega_0$, proving \eqref{eq2.20}.

\section{\spaceskip.33em plus.22em minus.17em Uniqueness of the fundamental solution
belonging to $\calO'_{\rm LOC}(H_+)$}\label{sec3}

\subsection{An associativity relation for convolution}\label{subsec3.1}

\begin{lem}\label{lem3.1} Suppose that {\rm(i)} holds. Fix
$\sigma\in\]\omega_0,\infty\[$ and define  $N_\sigma$ and $S_\sigma$ by 
\eqref{eq2.5} and \eqref{eq2.16}. Suppose moreover that
$U\in\calO'_{\rm LOC}(H_+)$ and that $P(\partial_0,\partial_1,\ldots,\partial_n)U$ has compact support. Then
\begin{equation}\label{eqnow3.1}
(N_\sigma*S_\sigma)*(e_{-\sigma}U)=N_\sigma*(S_\sigma*(e_{-\sigma}U)).
\end{equation}		
\end{lem}

\begin{proof}  Notice that both sides of \eqref{eqnow3.1} are well defined because every sign $*$ in \eqref{eqnow3.1} denotes a convolution of two distributions on $\symR^{1+n}$ one of which has compact support. To see this it is sufficient to observe that 
$\supp S_\sigma=\{0\}$, $N_\sigma*S_\sigma=S_\sigma*N_\sigma=\delta$, and, by \eqref{eqnow2.15},
$$
S_\sigma*(e_{-\sigma}U)=P(\sigma+\partial_0,\partial_1,\ldots,\partial_n)(e_{-\sigma}U)
=e_{-\sigma}(P(\partial_0,\partial_1,\ldots,\partial_n)U)
$$
has compact support. However, from the three factors $N_\sigma$, $S_\sigma$ and
$e_{-\sigma}U$ occurring in \eqref{eqnow3.1} only one has compact support, so that \eqref{eqnow3.1} does not follow from any of the simple criterions of the associativity of convolution. In order to prove that both sides of \eqref{eqnow3.1} are equal we will apply an argument going back to C.~Chevalley (\cite[pp.~120--121]{2}, proof of Theorem 2.2) which reduces the problem to the Fubini--Tonelli theorem.

Since the set $\{\varphi_1*\varphi_2*\varphi_3:\varphi_i\in\calD(\symR^{1+n})$ for
$i=1,2,3\}$ is dense in $\calD(\symR^{1+n})$, \eqref{eqnow3.4} will follow once it is proved that
\begin{equation}\label{eqnow3.2}
[(N_\sigma*S_\sigma)*(e_{-\sigma} N)]*[\varphi_1*\varphi_2*\varphi_3]=
[N_\sigma*(S_\sigma*(e_{-\sigma} N))]*[\varphi_1*\varphi_2*\varphi_3]
\end{equation}	    
for every $\varphi_1,\varphi_2,\varphi_3\in\calD(\symR^{1+n})$. 
In order to  prove  \eqref{eqnow3.2}, fix $\varphi_1$, $\varphi_2$, $\varphi_3$ and let
$$
f=N_\sigma*\varphi_1,\quad g= S_\sigma*\varphi_2,\quad h=(e_{-\sigma}U)*\varphi_3.
$$
Then $f,g,h\in C^\infty(\symR^{1+n})$ and using commutativity and associativity of convolution of distributions when all factors except at most one have compact support, one can prove that
\begin{equation}\label{eqnow3.3}
[(N_\sigma*S_\sigma)*(e_{-\sigma} U)]*[\varphi_1*\varphi_2*\varphi_3]=(f*g)*h
\end{equation}		   
and
\begin{equation}\label{eqnow3.4}
[N_\sigma*(S_\sigma*(e_{-\sigma} U))]*[\varphi_1*\varphi_2*\varphi_3]=f*(g*h).
\end{equation}
Let us stress that in the proof of \eqref{eqnow3.3} and \eqref{eqnow3.4} (and in particular in the proof that the right sides of \eqref{eqnow3.3} and \eqref{eqnow3.4} make sense) we have to make use of the facts that $N_\sigma*S_\sigma=\delta$  and $S_\sigma*(e_{-\sigma} U)$
has compact support. The equalities \eqref{eqnow3.3} and \eqref{eqnow3.4} reduce the problem of proving \eqref{eqnow3.2} to proving the equality
\begin{equation}\label{eqnow3.5}
(f*g)*h=f*(g*h).
\end{equation}
To do this, we need some more detailed information about $f$, $g$,~$h$. Since 
$N_\sigma\in\calO'_C(\symR^{1+n})$, by \eqref{eq1.1} one has
\begin{equation}\label{eqnow3.6}
f\in\calS(\symR^{1+n})\subset L^{1}(\symR^{1+n}).
\end{equation}
Since $\supp S_\sigma=\{0\}$, one has
\begin{equation}\label{eqnow3.7}
g\in\calD(\symR^{1+n})\subset L^{1}(\symR^{1+n}).
\end{equation}
Furthermore
\begin{equation}\label{eqnow3.8}
h\in C^\infty(\symR;\calS(\symR^n))\subset C(\symR; L^{1}(\symR^{n})).
\end{equation}

Indeed, for the proof of \eqref{eqnow3.8} it is sufficient to show that 
$[(e_{-\sigma} U)*\varphi_3]|_{[-a,a]\times\symR^n}\in C^\infty([-a,a];\calS(\symR^n))$
for every $a\in\]0,\infty\[$. So, take $a\in\]0,\infty\[$ and
$b\in\]0,\infty\[$
such that $\supp\varphi_3\subset[-b,b]\times\symR^n$. Take $\vartheta\in\calD(\symR)$
such that $\vartheta=1$  on $[-a-b,a+b]$. Then
\begin{equation}\label{eqnow3.9}
[(e_{-\sigma} U)*\varphi_3]|_{[-a,a]\times\symR^n}=[(\vartheta_{0} e_{-\sigma} U)*\varphi_3]|_{[-a,a]\times\symR^n}.
\end{equation}
Since $\vartheta_{0} e_{-\sigma} U\in\calO'_C(\symR^{1+n})$, by \eqref{eq1.1} one has 
$(\vartheta e_{-\sigma} U)*\varphi_3\in \calS(\symR^{1+n})$, so that \eqref{eqnow3.9} implies \eqref{eqnow3.8}.

Since $\supp N_\sigma, \supp e_{-\sigma} U\subset H_+$ there is
$c\in\]0,\infty\[$ (depending on $\varphi_1$, $\varphi_2$, $\varphi_3$, which however are fixed) such that
\begin{equation}\label{eqnow3.10}
\supp f,\supp g,\supp h\subset\{(x_0,x_1,\ldots,x_n)\in \symR^{1+n}:x_0\ge -c\}.
\end{equation}
From \eqref{eqnow3.6}--\eqref{eqnow3.8} and \eqref{eqnow3.10} it follows that
$(|f|*|g|)*|h|\in C([-3c,\infty\[$; $L^{1}(\symR^n))$. Hence
\begin{multline*}
\int_{\symR^{1+n}}\bigg(\int_{\symR^{1+n}}|f(v_0,\ldots,v_n)|\,|g(u_0-v_0,\ldots,u_n-v_n)|
\, dv_{0}\ldots dv_n\bigg)\\
\cdot |h(x_0-u_0,\ldots,x_n-u_n)|\,
du_0\ldots du_n<\infty
\end{multline*}
for every $(x_0,\ldots,x_n)\in \symR^{1+n}$, so that, by the Fubini--Tonelli theorem, the two iterated integrals corresponding to the integral
\begin{multline*}
\int_{\symR^{1+n}\times \symR^{1+n}} f(v_0,\ldots,v_n) g(u_0-v_0,\ldots,u_n-v_n)
h(x_0-u_0,\ldots,x_n-v_n)\\
 dv_{0}\ldots dv_{n}\,du_{0}\ldots du_{n}
\end{multline*}
are equal for every $(x_0,\ldots,x_n)\in \symR^{1+n}$. This means that \eqref{eqnow3.5} holds.
\end{proof}

\subsection{Uniqueness as a consequence of the associativity relation (\ref{eqnow3.1})}\label{subsec3.2}
The uniqueness of the fundamental solution belonging to $\calO'_{\rm LOC}(H_+)$ for the operator  $P(\partial_0,\partial_1,\ldots,\partial_n)$
satysfying (i) is a consequence of the following lemma.

\begin{lem}\label{lem3.2}  Suppose that {\rm(i)} holds and that $F\in\calE'(\symR^{1+n})$ has support contained in $H_+$. Then the equation
\begin{equation}\label{eqnow3.11}
P(\partial_0,\partial_1,\ldots,\partial_n)U=F
\end{equation}			    
has exactly one solution $U$ belonging to $\calO'_{\rm LOC}(H_+)$. Moreover, for this
solution and every $\sigma\in \]\omega_0,\infty\[$  one has
\begin{equation}\label{eqnow3.12}
U=(e_\sigma N_\sigma)*F.
\end{equation}
\end{lem}

\begin{proof}  Suppose that (i) holds. Take $\sigma\in\]\omega_0,\infty\[$. Then, in
view of (2.18) and \eqref{eq2.18}, $N=e_\sigma N_\sigma$ belongs to
$\calO'_{\rm LOC}(H_+)$ and is a fundamental solution for $P(\partial_0,\partial_1,\ldots,\partial_n)$.
It follows that  $U$ defined by \eqref{eqnow3.12} belongs to
$\calO'_{\rm LOC}(H_+)$ and satisfies \eqref{eqnow3.11}. It remains to prove that in
$\calO'_{\rm LOC}(H_+)$  there are no other solutions of \eqref{eqnow3.11}. To this end suppose that $U\in\calO'_{\rm LOC}(H_+)$ and $U$ satisfies \eqref{eqnow3.11}. Take
$\sigma\in\]\omega_0,\infty\[$ and define $S_\sigma$  by \eqref{eq2.16}. Then, by \eqref{eqnow2.15}, 
$$
S_\sigma*(e_{-\sigma} U)\!=\!P(\sigma+\partial_0,\partial_1,\ldots,\partial_n)(e_{-\sigma} U)
\!=\!e_{-\sigma}(P(\partial_0,\partial_1,\ldots,\partial_n)U)
\!=\!e_{-\sigma} F,
$$
whence, by \eqref{eq2.16}, \eqref{eq2.7} and \eqref{eqnow3.1},
\begin{align*}            
e_{-\sigma} U&=\delta*(e_{-\sigma} U)=(N_\sigma* S_\sigma)*(e_{-\sigma} U)
=N_\sigma*(S_\sigma*(e_{-\sigma} U))\\
&=N_\sigma*(e_{-\sigma} F)\!=\!e_{-\sigma}((e_{\sigma}N_\sigma)*F)
\end {align*}                                                
so that $U=(e_{\sigma}N_\sigma)*F$.
\end{proof}

\section{Proof of (ii)$\Rightarrow$(i)}\label{sec4}

\subsection{The distributions $\vartheta_0 N(\varphi\otimes{\cdot})$}\label{subsec4.1}
Let $N\in\calO'_{\rm LOC}(H_+)$ be a fundamental solution for $P(\partial_0,\partial_1,\ldots,\partial_n)$. 
Fix $a,b$ such that $0<a<b<\infty$, and $\vartheta\in\calD(\symR)$
such that $\vartheta=1$ on $[-b,b]$. For every $\varphi\in\calD(\symR)$ consider the mapping 
$$ 
T(\varphi):\calD(\symR^n)\ni\phi\mapsto\<\vartheta_0N,\varphi\otimes\phi\>\in \symC.
$$
Then $T(\varphi)\in\calD'(\symR^n)$.
Since 
$\vartheta_0N\in\calO^{\prime}_C
(\symR^{1+n})$, from \eqref{eq1.2} it follows that for every
$k\in\symN_0$ there is $m_k\in \symN_0$ such that
\begin{equation}\label{eq4.1}
\vartheta_0N=\sum_{p+|\alpha|\le m_k}\partial^p_0 \partial^{\alpha_1}_1\cdots \partial^{\alpha_n}_n F_{k;p,\alpha}
\end{equation}
where every $F_{k;p,\alpha}$ is a continuous function on
$\symR^{1+n}=\{(t,x)\in\symR\times\symR^{n}\}$
for which
$$
\sup_{(t,x)\in\symR^{1+n}}(1+|t|+|x|)^k|F_{k;p,\alpha}(t,x)|<\infty.
$$
Consequently, whenever $\varphi\in\calD(\symR)$, then
\begin{equation}\label{eq4.2}
T(\varphi)=\sum_{|\alpha|\le m_k}\partial_1^{\alpha_1}\cdots \partial_n^{\alpha_n}
f_{k;\alpha;\varphi}
\end{equation}
where
$$
f_{k;\alpha;\varphi}(x)=\sum_{p\le m_k-|\alpha|} \int_{\symR}((-\partial_0)^p\varphi(t)) F_{k;p;\alpha}(t,x)\,dt.
$$
It follows that, whenever $|\alpha|\le m_k$, $\varphi\in\calD(\symR)$, and 
$x\in\symR^{n}$, then
\begin{align}\label{eq4.3}
|f_{k;\alpha;\varphi}(x)|
&\le C_k \sum_{p\le m_k-|\alpha|} \int_{\supp\vartheta}
|\partial_0^p\varphi(t)|(1+|t|+|x|)^{-k}\, dt\\
&\le D_k(1+|x|)^{-k}\sup\{|\partial_0^p\varphi(t)|:p=0,\ldots,m_{k},\,t\in\symR\},\notag
\end{align}
where $C_k,D_k\in\]0,\infty\[$ depend only on $k$. In particular this shows that
\begin{equation}\label{eq4.4}
T(\varphi)\in\calO^{\prime}_C(\symR^{n})
\ \quad\mbox{for every }\varphi\in\calD(\symR).^{*)}
\end{equation}\footnote{$^{*)}$ After introducing the topology in $\calO^{\prime}_C(\symR^{n})$, it is possible to prove that the mapping
$\calD(\symR)\ni\varphi\mapsto T(\varphi)\in\calO^{\prime}_C(\symR^{n})$ is a vector-valued distribution. However this is insignificant for the present proof.}Since $N$ is the fundamental solution for $P(\partial_0,\partial_1,\ldots,\partial_n)$ with support in $H_+$, and $\vartheta=1$ on $[-b,b]$, it follows that
\bnumer
\item\label{eq4.5}
$T(\varphi)=0$ whenever $\supp\varphi\subset\]-\infty,0\[$,
\item\label{eq4.6} $\sum_{k=0}^mQ_k(\partial_1,\ldots,\partial_n)
T((-\partial_0)^k\varphi)
=\varphi(0)\delta$ for all $\varphi\in C^{\infty}_{[-b,b]}(\symR)$
where $\delta$ is the Dirac distribution on $\symR^n$
and $Q_k(\partial_1,\ldots,\partial_n),\,k=0,\ldots,m$, are PDOs on
$\symR^n$ such that $P(\partial_0,\partial_1,\ldots,\partial_n)
=\sum_{k=0}^m\partial_0^kQ_k (\partial_1,\ldots,\partial_n)$.
\enumer

In the subsequent lemmas it will be tacitly assumed that (ii) holds and $N$, $a$, $b$,
$\vartheta$, $T$ are fixed. 
Recall that $0<a<b<\infty$, $\vartheta\in\calD(\symR)$, $\vartheta=1$ on $[-b,b]$, $N\in\calO'_{\rm LOC}(H_{+})$ is a fundamental solution for $P(\partial_0,\partial_1,\ldots,\partial_n)$ and
$T(\varphi)=\vartheta_{0}N(\varphi\otimes\cdot)\in \calO'_{C}(\symR^{n})$ for every $\varphi\in\calD(\symR)$.
For every $\varphi\in\calD(\symR)$ denote by
$\widehat T(\varphi)$ the image of $T(\varphi)$ under the Fourier transformation on $\symR^{n}$. Then $\widehat T(\varphi)\in\calO_M(\symR^{n})$, by \eqref{eq4.4} and \eqref{eq1.3}.

\begin{lem}\label{lem4.1}
There are $p_0, m_0\in\symN_0$ and $C\in\]0,\infty\[$ such that
$$
|\widehat T(\varphi)(\xi)|\le C(1+|\xi|)^{m_0}\sup
\{|\partial_0^p\varphi(t)|:p=0,\ldots,p_0,\,a\le t\le b\}
$$
for every $\xi\in\symR^n$ and $\varphi\in C_{[a,b]}^\infty(\symR)$.
\end{lem}

\begin{proof} If in \eqref{eq4.1} we take $k> n$, then, by \eqref{eq4.2} and \eqref{eq4.3},
$$ 
T(\varphi)=\sum_{|\alpha|\le m_k} \partial_1^{\alpha_1}\cdots\partial_n^{\alpha_n}f_{k;\alpha;\varphi}
\ \quad\mbox{for every }\varphi\in C_{[a,b]}^\infty(\symR)
$$
where
$$
\|f_{k;\alpha;\varphi}\|_{L^1(\symR^{n})}\le D \sup\{|\partial_0^p\varphi(t)|:p=0,\ldots,m_k,\, a\le t\le b\}
$$
for every $\alpha$ with $|\alpha|\le m_k$ and every $\varphi\in
C_{[a,b]}^\infty(\symR)$, with $D\in\]0,\infty\[$ depending only on $k$. Consequently, whenever $\varphi\in
C_{[a,b]}^\infty(\symR)$, then
$$
|\widehat T(\varphi)(\xi)|\le(1+|\xi|)^{m_k}|g_{\varphi}(\xi)|_{M_{m\times m}} \ \quad\mbox{for every }\xi\in\symR^n
$$ 
where $g_\varphi\in C_b(\symR^{n})$ and
$$
\sup_{\xi\in\symR^{n}}|g_{\varphi}(\xi)|\le C\sup\{|\partial_0^p\varphi(t)|:
p=0,\ldots,m_k,\, a\le t\le b\}
$$
for some $C\in\]0,\infty\[$ depending only on $k$. 
\end{proof}

\subsection{An inequality of Chazarain type}\label{subsec4.2}

\begin{lem}\label{lem4.2}
Suppose that $P(\partial_0,\partial_1,\ldots,\partial_n)$ is a PDO on 
$\symR^{1+n}$ with constant coefficients for which there is a fundamental solution belonging to $\calO'_{\rm LOC}(H_+)$. Then there are $a,b\in
\]0,\infty\[$
such that whenever $(\lambda,\xi)\in\symC\times\symR^n$ and
$$
\hRe\lambda> a+b\log(1+|\lambda|+|\xi|)~{^{*)}},
$$\footnote{$^{*)}$ This inequality and its proof are similar to the inequality (1.2) 
on p.~394 of \cite{1} and the argument presented
on
p.~395 of \cite{1}. There is however an important difference. In \cite{1} the inequality (1.2) does not involve $\xi$ and determines the ``logarithmic region'' $\Lambda\subset\symC$ such that for every $\lambda\in\Lambda$ an abstract operator $Q(\lambda)=\lambda^m A_m+\cdots+\lambda A_1+A_0$
is invertible. In our case the inequality involves $\xi$ but the operator $Q(\lambda)$
is replaced by the \emph{polynomial} $P(\lambda,i\xi)$, and Lemma~\ref{lem4.2} is not the final step of the argument.}
then  $P(\lambda,i\xi)\ne0$.
\end{lem}

\begin{proof} From \eqref{eq4.6} it follows that
\begin{equation}\label{eq4.7}
\sum_{k=0}^m Q_k(i\xi)\widehat T((-\partial_0)^k\varphi)=\varphi(0)
\end{equation}
for every $\varphi\in C^\infty_{[-b,b]}(\symR)$
and $\xi\in\symR^n$.
Take $\varphi_0\in C^\infty_{[-b,b]}(\symR)$ such that $\varphi_0=1$ on $[-a,a]$. 
Following J.~Chazarain~\cite{1}, pp.~394--395, consider functions of the form
$\varphi=e_{-\lambda}\varphi_0$ where $\lambda$ ranges over $\symC$. 
Since $T(\varphi)=0$ whenever $\supp\varphi\subset\]-\infty,0\[$,
by
\eqref{eq4.7} and the Leibniz formula
one has 
\begin{align}\label{eq4.8}	
 P(\lambda,i\xi)[\widehat T(e_{-\lambda}\varphi_0)](\xi)
&=\Big(\sum_{k=0}^m \lambda^k Q_k(i\xi)\Big)[\widehat T(e_{-\lambda}\varphi_0)](\xi)\\
&=1-\sum_{k=0}^m Q_k(i\xi)[\widehat T(\psi_{k,\lambda})](\xi)\notag
\end{align}                      
where $\psi_{k,\lambda}\in C_{[a,b]}^\infty(\symR)$  is determined by the equality
$$
\psi_{k,\lambda}(t)=(-1)^k\sum_{j=1}^k{k\choose j}\partial_0^j
\varphi_0(t)(-\lambda)^{k-j} e^{-\lambda t}\ \quad\mbox{for }
t\in[a,b].
$$
By Lemma \ref{lem4.1} there are $C,K\in\]0,\infty\[$ such that, if 
$\hRe\lambda\ge 0$, then
\begin{multline}\label{eq4.9}	
|[\widehat T(\psi_{k,\lambda})](\xi)|\\
\begin{aligned}
&\le C(1+|\xi|)^{m_0}\sup\{|\partial_0^p\psi_{k,\lambda}(t)|:
p=0,\ldots,p_0,\,a\le t\le b\}\\
&\le C(1+|\xi|)^{m_0} K(1+|\lambda|)^{m-1+p_0} e^{-a\hRe \lambda}
\end{aligned} \indent
\end {multline}      
for every $k=1,\ldots,m$ and $\xi\in\symR^n$. Furthermore, there are
$l\in\symN$ and $L\in\]0,\infty\[$ such that
\begin{equation}\label{eq4.10}
\sum_{k=0}^m|Q_k(i\xi)|\le L(1+|\xi|)^l \ \quad\mbox{for every }
\xi\in\symR^n.
\end{equation}
Let $\mu=m_0+l+m-1+p_0$. From \eqref{eq4.8}--\eqref{eq4.10} it follows that if  $(\lambda,\xi)\in\symC\times \symR^n$, $\hRe\lambda\ge0$, and
$$
CKL(1+|\lambda|+|\xi|)^\mu e^{-a\hRe\lambda}<1,
$$
then $|\sum_{k=0}^m Q_k(i\xi)[\widehat T(\psi_{k,\lambda})](\xi)|<1$, and hence $P(\lambda,i\xi)\ne0$. Therefore, if $(\lambda,\xi)\in \symC\times \symR^n$ and
$$
\hRe\lambda>a^{-1}\log(CKL+1)+a^{-1}\mu\log(1+|\lambda|+|\xi|),
$$
then $P(\lambda,i\xi)\ne0$.
\end{proof}

\subsection{The  Chazarain type inequality implies (i)}\label{subsec4.3}
The implication (ii)$\Rightarrow$(i) is an immediate consequence of Lemma \ref{lem4.2} and the following

\begin{lem}\label{lem4.3}
Let $Q$ be a polynomial of $1+n$ variables with complex coefficients. Suppose that there are $a\in\symR$ and $b\in\]0,\infty\[$ such that
\begin{multline}\label{eq4.11}
\hRe\lambda\le a+b\log(1+|\lambda|+|\xi|)\\
\mbox{whenever }
(\lambda,\xi)\in \symC\times\symR^n\ \mbox{and}\ Q(\lambda,\xi)=0.
\end{multline}
Then
$$
\sup\{\hRe\lambda:(\lambda,\xi)\in \symC\times\symR^n,\,Q(\lambda,\xi)=0\}<\infty.
$$
\end{lem}

The proof follows the scheme due to L.~G{\aa}rding and L.~H\"ormander. Let 
\begin{multline*}
\sigma(r)=\sup\{\hRe\lambda:\lambda \in \symC\mbox{ and there is  }\xi\in\symR^n
\mbox{ such that} \\
|\lambda^2|+|\xi^2|\le\tfrac12 r^2\mbox{ and }Q(\lambda,\xi)=0\}.
\end{multline*}
Then, by \eqref{eq4.11},
\begin{equation}\label{eq4.12}
\sigma(r)\le a+b\log(1+r)\ \quad\mbox{for every }r\in[0,\infty\[.
\end{equation}
Following an idea of L.~H\"ormander (presented in the Appendix to \cite{6}), the Tarski--Seidenberg theorem is used to show that there is a polynomial $V(z,w)$ (not vanishing identically)
of two variables such that $V(r,\sigma(r))=0$ for every $r\in[0,\infty\[$. Then, as in L.~G{\aa}rding's proof of the Lemma on p.~11 of \cite{3}, the Puiseux expansions of the
$w$-roots of $V(z,w)$ for large $|z|$ show that \eqref{eq4.12} is possible only if 
$\sup\{\sigma(r):r\in[0,\infty\[\}<\infty$.



\newpage
\section*{Appendix}

\setcounter{equation}{0}
\renewcommand\theequation{A.{\arabic{equation}}}

\subsection*{Proof of Lemma 4.3}

Let $R(\sigma,\tau,\xi)$ and  $S(\sigma,\tau,\xi)$ be real polynomials on $\symR^{2+n}$ such that
$$
R(\sigma,\tau,\xi)+iS(\sigma,\tau,\xi)=Q(\sigma+i\tau,\xi).
$$
Then
\begin{multline*}
E=\{(r,\sigma,\tau,\xi)\in\symR^{3+n}:r\ge0,\,\sigma^2+\tau^2+|\xi|^2
\le\tfrac12 r^{2},\\
R(\sigma,\tau,\xi)=0,\,
S(\sigma,\tau,\xi)=0\}
\end{multline*}
is a semi-algebraic subset of $\symR^{3+n}$ and, by the Tarski--Seidenberg theorem (see Appendix to [H2]) its projection  on $\symR^{2}$ defined by
$$
F=\{(r,\sigma)\in\symR^{2}:\exists_{\tau,\xi}\ (r,\sigma,\tau,\xi)\in E\}
$$
is a semi-algebraic subset of  $\symR^{2}$. 
If $\sigma(r)$ is defined as in Sec.~4.3, then for
 every $r\in[0,\infty\[$ one has
\begin{equation}\label{eqA.1}
\sigma(r)=\sup\{\sigma:(r,\sigma)\in F\}.
\end{equation}
Since $F$ is semi-algebraic, it may be represented in the form
\begin{equation}\label{eqA.2}
F=\bigcup_{i=1}^k F_i\cap G_{i,1}\cap\cdots\cap G_{i,j(i)} 
\end{equation}	
where
\begin{equation}\label{eqA.3}
F_i=\{x\in\symR^{2}: P_i(x)=0\},\ \quad  G_{i,j}=\{x\in\symR^{2}: Q_{i,j}(x)>0\},
\end{equation}	
$P_i$ and $Q_{i,j}$ being real polynomials on $\symR^{2}$. It is not excluded that some $P_i$
are identically zero and some $Q_{i,j}$ are strictly positive on the whole $\symR^{2}$. 
From \eqref{eqA.1} it follows that whenever $r\in[0,\infty\[$ is fixed, there is $i(r)\in\{1,\ldots,k\}$ such that	
\begin{equation}\label{eqA.4}
\sigma(r)=\sup\{\sigma:(r,\sigma)\in F_{i(r)}\cap G_{i(r),1}\cap\cdots\cap G_{i(r),j(i(r))} \}.
\end{equation}
By \eqref{eq4.12}, for every $r\in[0,\infty\[$ one has $\sigma(r)<\infty$, so that there is a bounded sequence
$(\sigma_\nu(r))_{\nu=1}^{\infty}$ such that
\begin{equation}\label{eqA.5}
(r,\sigma_\nu(r))\in F_{i(r)}\cap G_{i(r),1}\cap\cdots\cap G_{i(r),j(i(r))}\ \quad
\mbox{for every }\nu=1,2,\ldots
\end{equation} 
and
\begin{equation}\label{eqA.6}
\lim_{\nu\to\infty}\sigma_\nu(r)=\sigma(r).
\end{equation}
If $P_{i(r)}\not\equiv 0$, then \eqref{eqA.5} and \eqref{eqA.6} imply that $P_{i(r)}(r,\sigma(r))=0$. If $P_{i(r)}\equiv\nobreak 0$, then, again by \eqref{eqA.5} and \eqref{eqA.6}, for some $j_0\in\{1,\ldots,j(i(r))\}$ one has $Q_{i(r),j_0}\not\equiv 0$ and 
$Q_{i(r),j_0}(r,\sigma(r))=0$, because otherwise $F_{i(r)}=\symR^2$ and there would be
$\varepsilon>0$ such that $Q_{i(r),j}(r,\sigma(r)+\varepsilon)>0$ for every 
$j\in\{1,\ldots,j(i(r))\}$ contrary to \eqref{eqA.4}. Consequently, whenever 
$r\in[0,\infty\[$, then either $W_r\equiv P_{i(r)}$ or $W_r\equiv Q_{i(r),j_0}$ is a real polynomial on $\symR^2$ such that
$$
W_r\not\equiv 0\quad\mbox{and}\quad  W_r(r,\sigma(r))=0.
$$
Therefore if $V$  is equal to the product of all those polynomials $P_{i}$ and $Q_{i,j}$, that occur in 
\eqref{eqA.3} and do not vanish identically on $\symR^2$, then
\begin{equation}\label{eqA.7}
V \not\equiv 0\quad\mbox{and}\quad V(r,\sigma(r))=0\quad\mbox{for every }
r\in[0,\infty\[.
\end{equation}

Now we are going to show that \eqref{eq4.12} and \eqref{eqA.7} imply
$\sup\{\sigma(r):r\in[0,\infty\[\}<\infty$. To this end we consider $V$ as a polynomial $V(z,w)$ of two complex veriables, and, following L.~G\aa rding~[G, proof of the Lemma on p.~11], we use the Puiseux expansions of the $w$-roots of $V(z,w)$. Concerning these expansions we will give exact references to \cite{S-Z}. Consider the factorization
$$
V(z,w)=V_1(z,w)\cdot V_2(z,w)\cdot\ldots\cdot V_l(z,w),
\quad z\in\symC\setminus\bigcup_{k=1}^l S_k,
\quad x\in\symC,
$$
where
\begin{itemize}
\item[(i)] every $V_k$, $k=1,\ldots,l$, belongs to the ring $K(z)[w]$ of polynomials of $w$
over the field $K(z)$ of rational functions of $z$, so that
$$
V_k(z,w)=\sum_{j=0}^{d_k}A_{k,j}(z)w^j\ \quad
\mbox{for every }z\in\symC\setminus S_k\
\mbox{and}\ x\in\symC
$$
where $A_{k,j}\in K(z)$ for $j=0,\ldots, d_k$, $A_{k,d_k}\not\equiv 0$, and the finite set $S_k$  consists of those points of $\symC$
 at which some $A_{k,j}$,  $j=0,\ldots, d_k$, has a pole,
\item[(ii)] every $V_k$, $k=1,\ldots,l$, is an irreducible element of $K(z)[w]$.
\end{itemize}
The assumption that $A_{k,d_k}\not\equiv 0$ implies that all the sets
$$
N_k=\{z\in\symC\setminus S_k:A_{k,d_k}(z)=0\},\ \quad
k=1,\ldots,l,
$$
are finite. Define
\begin{align*}
M_k&=\{z\in\symC\setminus (S_k\cup N_k):\mbox{not all the $w$-roots of
$V_k(z,w)$ are simple}\},\\
\calN_{k}&=\{(z,w)\in(\symC\setminus (S_k\cup N_k\cup M_k))\times \symC:V_k(z,w)=0\}.
\end {align*}
From Theorems VI.13.7, VI.14.2 and VI.14.3 of \cite{S-Z} it follows that
\begin{itemize}
\item[(a)] for every $k=1,\ldots,l$ the set $M_k$ is finite and
$$
\calN_k\cap[(\symC\setminus (S_k\cup N_k\cup M_k))\times \symC]
$$	                                  	
is equal to the graph of a $d_k$-variate function $\calR_k$ analytic on the set
$\symC\setminus (S_k\cup N_k\cup M_k)$,
\item[(b)] there is $R\in\]0,\infty\[$ such that for every $k=1,\ldots,l$ one has
$\{z\in\symC:R<|z|<\infty\}\subset\symC\setminus (S_k\cup N_k\cup M_k)$, 
and if $z\in\symC$ and $R<|z|<\infty$, then
$$
\calR_k(z)=\{\phi_k(\zeta):\zeta\in\symC,\,
0<|\zeta|< R^{-1/d_k},\,
\zeta^{d_k}=z^{-1}\}	                
$$
where $\phi_k$ is a function of one complex variable holomorphic in the annulus
$
\{\zeta\in\symC:0<|\zeta|< R^{-1/d_k}\},
$
\item[(c)] every $\phi_k$, $k=1,\ldots,l$, has at zero either a removable singularity or a pole.
\end{itemize}
Consequently, for every $k=1,\ldots,l$ one has
\begin{multline}\label{eqA.8}
\calN_k\cap (\{z\in\symC:|z|>R\}\times \symC)\\
=\Big\{\Big(z,\sum_{p=p_k}^{\infty}a_{k,p}\zeta^p\Big):(z,\zeta)\in\symC^2,\,
|z|>R,\, \zeta^{d_k}=z^{-1}\Big\}
\end{multline}	
where $\sum_{p=p_k}^\infty a_{k,p}\zeta^p$ is the Laurent expansion of $\phi_k$  in the annulus
$\{\zeta\in\symC:0<|\zeta|<R^{-1/d_k}\}$.
We assume that either $a_{k,p_k}\ne 0$ or $0=a_{k,p_k}=a_{k,p_k+1}=
\cdots.$ The equality \eqref{eqA.8} is nothing but the exact form of the Puiseux series expansion of $\calR_k(z)$ for  $z\to\infty$.
It follows that if $r\in\]R,\infty\[$, then $(r,\sigma(r))\subset \bigcup_{k=1}^l\calN_{k}$  and $\sigma(r)$ is equal to one of the numbers
$$
\sigma_{k,d}(r)=\sum_{p=p_k}^\infty a_{k,p}\bigg(\frac{e^{i2\pi d/d_k}}{\sqrt[\uproot3 {d_k}]{r}}\bigg)^p,\ \quad k=1,\ldots,l,\ d=1,\ldots,d_{k},
$$
where $\sqrt[\uproot3 {d_k}]{r_k}$ is the positive $d_k$-th root of $r$ and the series is absolutely convergent, so that $\sigma_{k,d}(r)=c_{k,d}r^{-p_k/d_k}(1+o(1))$
as $r\to\infty$ where $c_{k,d}= a_{k,p_{k}}e^{i2\pi d p_{k}/d_k}$. If for some $k=1,\ldots,l$
and $d=1,\ldots,d_k$ the set
$$
\{r\in\]R,\infty\[:\sigma(r)=\sigma_{k,d}(r)\}
$$
is unbounded, then $c_{k,d}$ must be real, and, by the estimation \eqref{eq4.12} of $\sigma(r)$, either  $c_{k,d}\le 0$, or $c_{k,d}>0$ and $p_k\ge 0$. In both cases
$\sup\{\sigma_{k,d}(r):r\in]R,\infty\[\}<\infty$. This implies that
$\sup\{\sigma(r):r\in[0,\infty\[\}<\infty$, completing the proof.

\end{document}